\documentclass[12pt]{article} 

\usepackage[utf8]{inputenc} 
\usepackage{geometry} 
\usepackage{graphicx} 
\usepackage{amsthm}

\usepackage{booktabs} 
\usepackage{array} 
\usepackage{paralist} 
\usepackage{verbatim} 
\usepackage{subfig} 
\usepackage[all]{xy}
\usepackage{amsfonts}
\usepackage{amssymb}
\usepackage{amsmath}
\numberwithin{equation}{section}

\usepackage{fancyhdr} 
\newtheorem{thm}{Theorem}[section]
\newtheorem{lemma}{Lemma}[section]

\newtheorem{remark}{Remark}[section]

\newtheorem*{proof of theorem1.1}{proof of theorem1.1}
\newtheorem*{proof of theorem1.2}{proof of theorem1.2}
\usepackage{sectsty}
\allsectionsfont{\sffamily\mdseries\upshape} 

\usepackage[nottoc,notlof,notlot]{tocbibind} 
\usepackage[titles,subfigure]{tocloft} 

\date{}
\begin{document}
\title{\large\textbf {Normalized solutions to  Schr\"{o}dinger systems with  potentials}\thanks{This work is supported by Shanghai Jiao Tong University Scientific and Technological Innovation Funds
partially supported by (NSFC-12031012, NSFC-11831003); and the Institute of Modern Analysis-A Frontier Research Center of Shanghai.}}
\author{Zhaoyang Yun$^{a,}$\thanks{Corresponding author}}
\maketitle
\centerline{\small $^{a}$ School of Mathematical Sciences, Shanghai Jiao Tong University,}
\centerline{\small Shanghai 200240, P. R. China, Email: yunzhaoyang@sjtu.edu.cn.}
\begin{abstract}

In this paper, we study the normalized solutions of the Schr\"{o}dinger system with trapping potentials
\begin{equation}\label{eq:diricichlet}
\begin{cases}
-\Delta u_1+V_1(x)u_1-\lambda_1 u_1=\mu_1 u_1^3+\beta u_1u_2^{2}+\kappa u_2~\hbox{in}~ \mathbb{R}^3,\\
-\Delta u_2+V_2(x)u_2-\lambda_2 u_2=\mu_2 u_2^3+\beta u_1^2u_2+\kappa u_1~\hbox{in}~ \mathbb{R}^3,\\
 u_1\in H^1(\mathbb{R}^3), u_2\in H^1(\mathbb{R}^3),\nonumber
\end{cases}
\end{equation}
under the constraint
\begin{equation}
\int_{\mathbb{R}^3} u_1^2=a_1^2,~\int_{\mathbb{R}^3} u_2^2=a_2^2\nonumber,
\end{equation}
where $\mu_1,\mu_2,a_1,a_2,\beta>0$, $\kappa\in\mathbb{R}$, $V_1(x)$ and $V_2(x)$ are trapping potentials, and $\lambda_1,\lambda_2$ are lagrangian multipliers, this is a typical $L^2$-supercritical case in $\mathbb{R}^3$. We obtain the existence of solutions to this system by minimax theory on the manifold for  $\kappa=0$ and $\kappa\neq 0$ respectively.
\end{abstract}
\section{Introduction}
We study the existence of solutions  for Schr\"{o}dinger systems with linear and nonlinear couplings
\begin{equation}\label{yun3-10-1}
\begin{cases}
-\Delta u_1+V_1(x)u_1-\lambda_1 u_1=\mu_1 u_1^3+\beta u_1u_2^{2}+\kappa u_2~\hbox{in}~\mathbb{R}^3,\\
-\Delta u_2+V_2(x)u_2-\lambda_2 u_2=\mu_2 u_2^3+\beta u_1^2u_2+\kappa u_1~\hbox{in}~\mathbb{R}^3,\\
 u_1\in H^1(\mathbb{R}^3), u_2\in H^1(\mathbb{R}^3),
\end{cases}
\end{equation}
which arise from Bose-Einstein condensates(see \cite{DKNF}). These systems have been concerned by many mathematicians in
recent years(see \cite{Z-1,Z-2,Z-3,Y-Z}), the existence of normalized solution $(\lambda_1,\lambda_2,u_1,u_2)\in \mathbb{R}^2\times H^1(\mathbb{R}^3)\times H^1(\mathbb{R}^3)$ is one of most important topics, i.e., solutions that satisfy the additional condition
\begin{equation}\label{yun3-10-2}
\int_{\mathbb{R}^3} u_1^2=a_1^2,~\int_{\mathbb{R}^3} u_2^2=a_2^2,
\end{equation}
where $a_1>0$ and $a_2>0$ are perscribed mass, $V_i(x)$, $i=1,2$ satisfy the following conditions $(A_1)$ and $(A_2)$, which are called trapp potentials.

({$A_1$})\ $V_i(x)\in C^1(\mathbb{R}^3,\mathbb{R})$, $\inf_{\mathbb{R}^3}V_i(x)=b_i>0$, $i=1,2$.

($A_2$)\ $\lim\limits_{|x|\rightarrow \infty}V_i(x)=\infty$, $i=1,2$.

When $\kappa=0$, the reason why we consider \eqref{yun3-10-1}-\eqref{yun3-10-2} is to find the standing wave of cubic Schr\"{o}dinger system with perscribed mass. More precisely, if $(\lambda_1,\lambda_2,u_1,u_2)\in \mathbb{R}^2\times H^1(\mathbb{R}^3)\times H^1(\mathbb{R}^3)$ is a solution of \eqref{yun3-10-1}-\eqref{yun3-10-2}, we take
$$\Phi_1(t,x)=e^{-\lambda_1t}u_1(x),~\Phi_2(t,x)=e^{-\lambda_2t}u_2(x),$$
then $(\Phi_1,\Phi_2)$ is the standing wave of following Schr\"{o}dinger system
\begin{equation}\label{yun3-10-3}
\begin{cases}
-\partial_t \Phi_1-\Delta \Phi_1+V_1(x)\Phi_1=\mu_1 \Phi_1^3+\beta \Phi_1\Phi_2^{2}~\hbox{in}~\mathbb{R}\times\mathbb{R}^3,\\
-\partial_t \Phi_2-\Delta \Phi_2+V_2(x)\Phi_2=\mu_2 \Phi_2^3+\beta \Phi_1^2\Phi_2~\hbox{in}~\mathbb{R}\times\mathbb{R}^3,\\
\end{cases}
\end{equation}
and the mass of $(\Phi_1,\Phi_2)$ is perscribed
$$\int_{\mathbb{R}^3}|\Phi_1|^2=a_1^2,~\int_{\mathbb{R}^3}|\Phi_2|^2=a_2^2,$$
along the trajectories of (1.3). When $V_i(x)=0$, $i=1,2$ the existence of solutions for \eqref{yun3-10-1}-\eqref{yun3-10-2} can be found in \cite{B-1}, \cite{B-2}, \cite{B-3}, \cite{B-4}, \cite{B-5}, \cite{B-6}. When $V_i(x)\neq 0$, $i=1,2$ the existence of solutions for \eqref{yun3-10-1}-\eqref{yun3-10-2} can be found in \cite {B-H} by bifurcation theory. Other interesting results for normalized solutions of Schrödinger equations or systems can be found in  \cite{G-J1}, \cite{G-Z1}, \cite{N-T}.
\par
When $\kappa\neq 0$, $\lambda_1$ and $\lambda_2$ are fixed, \eqref{yun3-10-1} is an important model for Schr\"{o}dinger systems, which was studied by many mathmaticians in \cite{Z-1}, \cite{Z-2}, \cite{Z-3}, \cite{Z-4}. When $V_i(x)=0$, $i=1, 2$ the existence of normalized solutions of  \eqref{yun3-10-1}-\eqref{yun3-10-2} can be found in \cite{Y-Z}. When $V_i(x)\neq 0$, $i=1, 2$ to the best of our knowledge there is no results of systems \eqref{yun3-10-1}-\eqref{yun3-10-2}. Inspired by these papers we are going to study the existence of solutions $(\lambda_1,\lambda_2,u_1,u_2)\in \mathbb{R}^2\times H^1(\mathbb{R}^3)\times H^1(\mathbb{R}^3)$ for system  \eqref{yun3-10-1}-\eqref{yun3-10-2}, where $\lambda_1$ and $\lambda_2$ are Lagrangian multipliers.
\par
 The weak solution of \eqref{yun3-10-1}-\eqref{yun3-10-2} is the critical point of the following functional
\begin{equation}\label{yun3-10-4}
\begin{split}
J(u_1,u_2)=&\frac{1}{2}\int_{\mathbb{R}^3}|\nabla u_1|^2+V_1(x)u_1^2+|\nabla u_2|^2+V_2(x)u_2^2\\
&-\frac{1}{4}\int_{\mathbb{R}^3}\mu_1 u_1^4+\mu_2 u_2^4+2\beta u_1^2u_2^2-\kappa\int_{\mathbb{R}^3}u_1u_2,
\end{split}
\end{equation}
on the manifold $S_1\times S_2$ of space $H_1\times H_2$, by regularity theory of elliptic equations, weak solutions of \eqref{yun3-10-1}-\eqref{yun3-10-2} are classical, where
$$S_1\times S_2:=\{(u_1,u_2)\in H_1\times H_2: |u_1|_2=a_1,|u_2|_2=a_2\},$$
 $H_1$ and $H_2$ are defined by
$$H_i:=\{u\in H^1(\mathbb{R}^3): \int_{\mathbb{R}^3}|\nabla u|^2+V_i(x)u^2<\infty\},~i=1,2,$$
and the $L^p$ norm denotes by
$$|u|_p:=(\int_{\mathbb{R}^3}|u|^p)^{\frac{1}{p}},~1\leq p\leq \infty.$$
Now we  state more basic assumptions on the potentials.

$(A_3)$\ There exists $D\in C(\mathbb{R},\mathbb{R}^+)$ such that
$$V_i(\alpha x)\leq D(\alpha)V_i(x),~\forall \alpha\in\mathbb{R},\ i=1,2,$$
where $\mathbb{R}^+:=[0,\infty).$

$(A_4)$\ There exists a constant $C_1>0$ such that
$$2V_i(x)+\nabla V_i(x)\cdot x\geq -C_1,\ i=1,2.$$

\begin{remark}
There are many functions which satisfy $(A_1)$, $(A_2)$, $(A_3)$ and $(A_4)$. For example
\par
(a) $V_i(x)=|x|^{\theta_{i}}+\delta_{i},~\theta_i>1,~\delta_i>0,~i=1,2.$
\par
(b) $V_i(x)=\ln (|x|^{\theta_{i}}+\eta_{i}),~\theta_i>1,~\eta_i>1,~i=1,2.$
\end{remark}
In order to find the weak solutions of \eqref{yun3-10-1}-\eqref{yun3-10-2} we need to overcome some difficulities, because the functional $J$ is unbounded from below on ${S_1\times S_2}$, we can not use Ekland variational principle to $J$ on $S_1\times S_2$ directly; and because $\lambda_1$ and $\lambda_2$ are unknown parts, we can not use  Nehari manifold method to find the solution.  To this end, we try to construct a minimax structure of $J$ on the manifold $S_1\times S_2$, by the minimax theory on the Finsler manifold which was introduced in \cite{GH},  we obtain a critical point of $J$ on $S_1\times S_2$, which is a weak solution of \eqref{yun3-10-1}-\eqref{yun3-10-2}.
\par
Note that $(A_1)$ and $(A_2)$  imply that the embedding $H_i\hookrightarrow L^p(\mathbb{R}^3)$ is compact when $2\leq p<2^*=6$ (the dimension is $N=3,2^*=\frac{2N}{N-2}$). Condition $(A_3)$ and $(A_4)$ are in order to make sure that the Palais-Smale sequence of $J|_{S_1\times S_2}$ we found in the following is bounded.
\par
Let $C_{a_1,a_2}:=\max\{(\mu_1+\beta)a_1G_4^4,(\mu_2+\beta)a_2G_4^4\}$, where $G_4$ denotes the best constant of Gagliardo-Nirenberg inequality which will be defined in Lemma \ref{le3-30-2}, $\lambda_{V_i}$, $i=1,2$ are defined by
$$\lambda_{V_i}:=\inf_{\phi\in H_i} \{ \int_{\mathbb{R}^3}|\nabla \phi|^2+V_i(x)\phi^2:\int_{\mathbb{R}^3}\phi^2=1\}.$$
\begin{remark}\label{re3-30-1}
By {\cite{N-T}} we know $\lambda_{V_i}>0$, $i=1,2$, are achieved by nonnegative functions which denote by $u_{V_i}$.
\end{remark}
\par
The main results  obtained by minimax methods read as follows.
\begin{thm}\label{th3-30-1}
Assume $(A_1)$, $(A_2)$, $(A_3)$, $(A_4)$ hold and $\kappa=0$, if $a_1$, $a_2$ satisfy that
\begin{equation}\label{6-1-2}
(\lambda_{V_1}a_1^2+\lambda_{V_2}a_2^2)C_{a_1,a_2}^2<\frac{16}{27},
\end{equation}
 then \eqref{yun3-10-1}-\eqref{yun3-10-2} has a solution $(\lambda_{1},\lambda_{2},u_{1,0},u_{2,0})$, and $u_{1,0}\neq 0$ and $u_{2,0}\neq 0$ are both nonnegative.
\end{thm}
Based on the proof of Theorem \ref{th3-30-1}, we have following existence theorems when $\kappa\neq 0$ by constructing the minimax structure on $S_1\times S_2$.
\begin{thm}\label{th3-30-2}
Assume $(A_1)$, $(A_2)$, $(A_3)$, $(A_4)$ hold and $0<\kappa<\frac{4}{27C^2_{a_1,a_2}a_1a_2}$, if $a_1$, $a_2$ satisfy that
\begin{equation}\label{6-1-1}
(\lambda_{V_1}a_1^2+\lambda_{V_2}a_2^2)C_{a_1,a_2}^2<\frac{8^2}{9^3},
\end{equation}then \eqref{yun3-10-1}-\eqref{yun3-10-2} has a solution $(\lambda_{1},\lambda_{2},u_{1,0},u_{2,0})$, and $u_{1,0}\neq 0$ and $u_{2,0}\neq 0$ are both nonnegative.
\end{thm}
Similarly, for $\kappa<0$ we have following theorem.
\begin{thm}\label{th3-30-3}
Assume $(A_1)$, $(A_2)$, $(A_3)$, $(A_4)$ hold and $-\frac{4}{27C^2_{a_1,a_2}a_1a_2}<\kappa<0$, if $a_1$, $a_2$ satisfy that
\begin{equation}\label{6-1-3}
(\lambda_{V_1}a_1^2+\lambda_{V_2}a_2^2)C_{a_1,a_2}^2<\frac{8^2}{9^3},
\end{equation}then \eqref{yun3-10-1}-\eqref{yun3-10-2} has a solution $(\lambda_{1},\lambda_{2},u_{1,0},u_{2,0})$, and $u_{1,0}\geq 0$, $u_{2,0}\leq 0$ or $u_{1,0}\leq 0$, $u_{2,0}\geq 0$.
\end{thm}
Moreover, by modifying the minimax structure in the proof of Theorem \ref{th3-30-2} and Theorem \ref{th3-30-3}, we can get some other existence results for $(1.1)$-$(1.2)$ when $\kappa\neq 0$, and $\kappa$ is independent of $a_1$ and $a_2$.
\begin{thm}\label{th3-30-4}
Assume $(A_1)$, $(A_2)$, $(A_3)$, $(A_4)$ hold and $0<\kappa<\sqrt{\inf_{\mathbb{R}^3}V_1\cdot\inf_{\mathbb{R}^3}V_2}$, if $a_1$, $a_2$ satisfy that
\begin{equation}\label{6-1-4}
(\lambda_{V_1}a_1^2+\lambda_{V_2}a_2^2)C_{a_1,a_2}^2<\frac{16}{27},
\end{equation}then \eqref{yun3-10-1}-\eqref{yun3-10-2} has a solution $(\lambda_1,\lambda_2,u_{1,0},u_{2,0})$, and $u_{1,0}\neq 0$ and $u_{2,0}\neq 0$ are both nonnegative.
\end{thm}
Similarly, for $\kappa<0$ we have the following theorem.
\begin{thm}\label{th3-30-5}
Assume $(A_1)$, $(A_2)$, $(A_3)$, $(A_4)$ hold and $-\sqrt{\inf_{\mathbb{R}^3}V_1\cdot\inf_{\mathbb{R}^3}V_2}<\kappa<0$, if $a_1$, $a_2$ satisfy that
\begin{equation}\label{6-1-5}
(\lambda_{V_1}a_1^2+\lambda_{V_2}a_2^2)C_{a_1,a_2}^2<\frac{16}{27},
\end{equation} then \eqref{yun3-10-1}-\eqref{yun3-10-2} has a solution $(\lambda_1,\lambda_2,u_{1,0},u_{2,0})$, and $u_{1,0}\geq 0$, $u_{2,0}\leq 0$ or $u_{1,0}\leq 0$, $u_{2,0}\geq 0$.
\end{thm}
\begin{remark}
Condition \eqref{6-1-2}, \eqref{6-1-1} are used to construct the minimax structures in Section 2 and Section 3  respectively.

\end{remark}
The paper is organized as follws.  In Section 2 we give some auxiliary results and study the existence of critical points of $J|_{S_1\times S_2}$ when $\kappa=0$. From \cite{B-2} we know that \eqref{yun3-10-1}-\eqref{yun3-10-2} is $L^2$-supercritical case, then $\inf_{S_1\times S_2} J$ can not be achieved, so we use the minimax theory on manifold which was introduced in \cite{GH} to get a minimax point of $J|_{S_1\times S_2}$. In Section 3 we study the existence of solutions for \eqref{yun3-10-1}-\eqref{yun3-10-2} when $\kappa\neq 0$, the minimax structure  constructed in Section 2 still holds, then we can get the conclusion which is similar to Section 2.
\section{The case $\kappa=0$}
First we give some results which will be useful in the following discussion.
\begin{lemma}\label{le3-30-1}
Assume $(A_1)$ and $(A_2)$ hold, then the embedding
$$H_i\hookrightarrow L^p(\mathbb{R}^3),$$
is compact when $2\leq p<2^*=6$.
\end{lemma}
\begin{proof}
This is a classical result, the proof can be found in \cite{T-W-M}.
\end{proof}
Next lemma is the well known Gagliardo-Nirenberg inequality, we will call it G-N inequality for simplicity, which is useful to construct the minimax structure of $J$ on the manifold $S_1\times S_2$.
\begin{lemma}\label{le3-30-2}[Gagliardo-Nirenberg inequality]
For any $u\in H^1(\mathbb{R}^3)$ we have
\begin{equation}\label{yun3-10-5}
|u|_p\leq G_{p} |\nabla u|^{\alpha}_2|u|^{1-\alpha}_2,\nonumber
\end{equation}
where $\alpha=\frac{3(p-2)}{2p}$.
\end{lemma}
\begin{remark}
By G-N inequality, we have the following useful inequality.
\begin{equation}\label{eq3-30-1}
\begin{split}
\int_{\mathbb{R}^3}\mu_1 u_1^4+\mu_2 u_2^4+2\beta u_1^2 u_2^2&\leq(\mu_1+\beta)\int_{\mathbb{R}^3} u_1^4+(\mu_2+\beta )\int_{\mathbb{R}^3}u_2^4\\
&\leq(\mu_1+\beta)a_1G_4^4|\nabla u_1|_2^3+(\mu_2+\beta)a_2G_4^4|\nabla u_2|_2^3\\
&\leq C_{a_1,a_2}(\int_{\mathbb{R}^3}|\nabla u_1|^2+|\nabla u_2|^2)^{\frac{3}{2}}\\
&\leq C_{a_1,a_2}(\int_{\mathbb{R}^3}|\nabla u_1|^2+V_1(x)u_1^2+|\nabla u_2|^2+V_2(x)u_2^2)^{\frac{3}{2}},
\end{split}
\end{equation}
 $\forall (u_1,u_2)\in H_1\times H_2$.
\end{remark}
Next we prove Theorem \ref{th3-30-1} by following steps.

\textbf{Step 1: }Construct the minimax structure of $J$ on manifold $S_1\times S_2$.

\textbf{Step 2: }Find a bounded PS sequence of $J|_{S_1\times S_2}$.

\textbf{Step 3: }Use compact embedding theorem to get the conclusion.

For this purpose, we introduce following sets to get the minimax structure:
$$A_K:=\{(u_1,u_2)\in S_1\times S_2:\int_{\mathbb{R}^3} |\nabla u_1|^2+V_1(x)u_1^2+\int_{\mathbb{R}^3}|\nabla u_2|^2+V_2(x)u_2^2\leq K\},$$
$$B_K:=\{(u_1,u_2)\in S_1\times S_2:\int_{\mathbb{R}^3} |\nabla u_1|^2+V_1(x)u_1^2+\int_{\mathbb{R}^3}|\nabla u_2|^2+V_2(x)u_2^2=\theta K\},$$
where $K>0$ and $\theta>1$ are fixed constants.
\par
Next lemma shows that there is a minimax structure on $A_K$ and $B_K$.
\begin{lemma}\label{le4-7-1}
Assume \eqref{6-1-2} holds and $\theta=3$, then there exists a $K>0$ such that
$$\inf_{B_K}J(u_1,u_2)>0,$$
and
$$\inf_{B_K}J>\sup_{A_K}J.$$
\end{lemma}
\begin{proof}
First we take $(u_1,u_2)\in B_K$, and choose $K$ such that
$$ \lambda_{V_1}a_1^2+\lambda_{V_2}a_2^2<K<\frac{16}{27C_{a_1,a_2}^2},$$ where $\frac{16}{27C_{a_1,a_2}^2}$ is the biggest zero point of function
$$ x-\frac{C_{a_1,a_2}}{4}(3 x)^{\frac{3}{2}},$$
and for $(u_1,u_2)\in B_K$ there exists $\delta_1>0$ such that
\begin{equation}\label{3-7-2}
\begin{split}
J(u_1,u_2)=&\frac{1}{2}\int_{\mathbb{R}^3} |\nabla u_1|^2+V_1(x)u_1^2+\frac{1}{2}\int_{\mathbb{R}^3}|\nabla u_2|^2+V_2(x)u_2^2\\
&-\frac{1}{4}\int_{\mathbb{R}^3}\mu_1 u_1^4+\mu_2 u_2^4+2\beta u_1^2u_2^2\\
\geq &\frac{3}{2}K-\frac{C_{a_1,a_2}}{4}(3 K)^{\frac{3}{2}}\\
>&\delta_1.
\end{split}
\end{equation}
Next we take $(v_1,v_2)\in B_K$ and $(u_1,u_2)\in A_K$ there exists $\delta_2>0$ such that
\begin{equation}\label{eq3-30-3}
\begin{split}
J(v_1,v_2)-J(u_1,u_2)=&\frac{1}{2}\int_{\mathbb{R}^3} |\nabla v_1|^2+V_1(x)v_{1}^2+|\nabla v_2|^2+V_2(x)v_{2}^2\\
&-\frac{1}{2}\int_{\mathbb{R}^3} |\nabla u_1|^2+V_1(x)u_{1}^2+|\nabla u_2|^2+V_2(x)u_{2}^2\\
&-\frac{1}{4}\int_{\mathbb{R}^3} \mu_1v_1^4+\mu_2 v_2^4+2\beta v_1^2 v_2^2+\frac{1}{4}\int_{\mathbb{R}^3} \mu_1 u_1^4+\mu_2 u_2^4+2\beta u_1^2u_2^2\\
\geq& \frac{1}{2}\int_{\mathbb{R}^3} |\nabla v_1|^2+V_1(x)v_{1}^2+|\nabla v_2|^2+V_2(x)v_{2}^2\\
&-\frac{1}{2}\int_{\mathbb{R}^3} |\nabla u_1|^2+V_1(x)u_{1}^2+|\nabla u_2|^2+V_2(x)u_{2}^2\\
&-\frac{1}{4}\int_{\mathbb{R}^3} \mu_1v_1^4+\mu_2 v_2^4+2\beta v_1^2 v_2^2\\
\geq& K-\frac{C_{a_1,a_2}}{4}(3K)^{\frac{3}{2}}\\
>&\delta_2.
\end{split}
\end{equation}
This finishes the proof.
\end{proof}
\begin{remark}
It is obvious that condition \eqref{6-1-2} is in order to ensure that $A_K$ is not empty, when $\lambda_{V_1}a_1^2+\lambda_{V_2}a_2^2<K\leq\frac{16}{27C_{a_1,a_2}^2}$. In fact, choose  $u_1=a_1u_{V_1}$ and $u_2=a_2u_{V_2}$ where $u_{V_i}$, $i=1,2$ are defined in Remark \ref{re3-30-1}.
\end{remark}
Now we use the transformation which was introduced in \cite{J-1}, in order to find a minimax structure of $J|_{S_1\times S_2}$,  we take  $s\in\mathbb{R}$ and define
$$(s\star u)(x):=e^{\frac{3}{2}}u(e^s x)\hbox{,}$$
it is clearly that
$$|s\star u|_2^2=|u|_2^2,~|\nabla(s\star u)|_2^2=e^{2s}|\nabla u|_2^2,$$
the condition $(A_3)$ ensure that, if $u\in H_i$, then $(s\star u)\in H_i$ for $i=1,2$, $s\in\mathbb{R}$. In fact
\begin{equation}
\begin{split}
\int_{\mathbb{R}^3}|\nabla (s\star u)|^2+V_i(x)(s\star u)^2&=e^{2s}\int_{\mathbb{R}^3}|\nabla u|^2+\int_{\mathbb{R}^3}V_i(e^{-s}x)u^2\\
&\leq e^{2s}\int_{\mathbb{R}^3}|\nabla u|^2+D(e^{-s})\int_{\mathbb{R}^3}V_i(x)u^2\\
&<\infty.\nonumber
\end{split}
\end{equation}
Now we can construct the minimax structure of $J|_{S_1\times S_2}$. First we fix a point $(v_1,v_2)\in A_K$ which is nonnegative and take
$w_i=(s*v_i)$, $i=1,2$ for $s\in\mathbb{R}^+$ large enough, by $(A_3)$ we have, $D(e^{-s})\rightarrow D(0)$ as $s\rightarrow \infty$, thus
\begin{equation}
\begin{split}
J(w_1,w_2)=&\frac{e^{2s}}{2}\int_{\mathbb{R}^3}|\nabla v_1|^2+|\nabla v_2|^2+\int_{\mathbb{R}^3}V_1(e^{-s}x)v_1^2+V_2(e^{-s}x)v_2^2\\
&-\frac{e^{3s}}{4}\int_{\mathbb{R}^3} \mu_1v_1^4+\mu_2v_2^4+2\beta v_1^2v_2^2\\
\leq{}&\frac{e^{2s}}{2}\int_{\mathbb{R}^3}|\nabla v_1|^2+|\nabla v_2|^2+D(e^{-s})\int_{\mathbb{R}^3}V_1(x)v_1^2+V_2(x)v_2^2\\
&-\frac{e^{3s}}{4}\int_{\mathbb{R}^3} \mu_1v_1^4+\mu_2v_2^4+2\beta v_1^2v_2^2\\
\rightarrow& -\infty,\nonumber
\end{split}
\end{equation}
as $s\rightarrow \infty$. Next we can fix $s\in\mathbb{R}$ large enough, such that $J(w_1,w_2)$ is negative enough. Moreover, any path from $(v_1, v_2)$ to $(w_1,w_2)$ must pass through $B_{K}$. In  fact
\begin{equation}\label{eq3-31-1}
\begin{split}
&\int_{\mathbb{R}^3}|\nabla (s\star v_1)|^2+V_1(x)(s\star v_1)^2+\int_{\mathbb{R}^3}|\nabla (s\star v_2)|^2+V_2(x)(s\star v_2)^2\\
\geq{}&\int_{\mathbb{R}^3}|\nabla (s\star v_1)|^2+\int_{\mathbb{R}^3}|\nabla (s\star v_2)|^2\\
={}&\frac{e^{2s}}{2}\int_{\mathbb{R}^3}|\nabla v_1|^2+|\nabla v_2|^2\\
\rightarrow {}&\infty,
\end{split}
\end{equation}
as $n\rightarrow\infty$.

Then we can find a minimax structure of $J|_{S_1\times S_2}$, take
\begin{equation}
\Gamma:=\{\gamma(t)=(\gamma_1(t),\gamma_2(t))\in C([0,1],S_1\times S_2):\gamma(0)=(v_1,v_2),\gamma(1)=(w_1,w_2)\},\nonumber
\end{equation}
note that every path from $(v_1,v_2)$ to $(w_1,w_2)$ must through $B_K$, then the minimax value of $J|_{S_1\times S_2}$ on $\Gamma$ can be defined as follows
\begin{equation}\label{eq3-31-2}
c:=\inf _{\gamma \in \Gamma}\sup_{t\in [0,1]}J(\gamma(t))\geq \inf_{B_K}J(u_1,u_2)>0.
\end{equation}
In order to obtain the boundedness of the PS sequence at level $c$, we use following notations.
\begin{equation}
\tilde{J}(s,u_1,u_2):=J(s\star u_1,s\star u_2)=\tilde J(0,s\star u_1,s\star u_2),\nonumber
\end{equation}
the corresponding minimax structure is as follows
$$\tilde {\Gamma}:=\{\tilde {\gamma}(t)=(s(t),\gamma_1(t),\gamma_2(t))\in C([0,1],\mathbb{R}\times S_1\times S_2):\tilde\gamma(0)=(0,v_1,v_2),\tilde\gamma(1)=(0,w_1,w_2)\},$$
and the minimax value is
$$\tilde c=\inf_{\tilde{\gamma} \in \tilde{\Gamma}}\sup_{t\in [0,1]}\tilde J(\tilde \gamma(t)).$$
It is easy to deduce that $\tilde c=c$. In fact, from $\tilde \Gamma \supset\Gamma$  we have $\tilde c\leq c$. On the other hand, for any $$\tilde\gamma(t)=(s(t),\gamma_1(t),\gamma_2(t)),$$
by definition we have
$$\tilde J(\tilde\gamma(t))=J(s(t)\star\gamma(t)),$$
and $s(t)\star \gamma (t)\in \Gamma$ is obvious, then
$$\sup_{t\in[0,1]}\tilde J(\tilde\gamma(t))\geq \inf _{\gamma \in \Gamma}\sup_{t\in [0,1]}J(\gamma(t)),$$
by defintion of $\tilde c$ we have $\tilde c\geq c$, then $\tilde c=c$.
Because
$$\tilde J(s,|u_1|,|u_2|)= \tilde J(s,u_1,u_2)=\tilde J(0,s\star u_1,s\star u_2),$$
 we can take a sequence $\tilde \gamma_n=(0,\gamma_{1,n},\gamma_{2,n})\in\tilde \Gamma$
and $\gamma_{1,n},\gamma_{2,n}\geq 0$ such that
$$c=\lim_{n\rightarrow\infty}\sup_{t\in[0,1]}\tilde J(\tilde \gamma_n(t)).$$
(Here we point out that when $\kappa>0$ we have
$$\tilde J(s,|u_1|,|u_2|)\leq\tilde J(s,u_1,u_2),$$ we can also take a sequence $\tilde \gamma_n=(0,\gamma_{1,n},\gamma_{2,n})\in\tilde \Gamma$
and $\gamma_{1,n},\gamma_{2,n}\geq 0$ such that
$$c=\lim_{n\rightarrow\infty}\sup_{t\in[0,1]}\tilde J(\tilde \gamma_n(t)).$$
Similarly, we can choose $\gamma_{1,n}\leq 0$, $\gamma_{2,n}\geq 0$ or $\gamma_{1,n}\geq 0$, $\gamma_{2,n}\leq 0$ when $\kappa<0$.)
\par
Then by Theorem 3.2 of \cite{GH} we can get a PS sequence $(s_n,\tilde u_{1,n},\tilde u_{2,n})$ of $\tilde J$ on $\mathbb{R}\times S_1\times S_2$ at level c. Moreover, we have
$$
\lim_{n\rightarrow\infty}|s_n|+dist_{H_i}((\tilde u_{1,n},\tilde u_{2,n}),(\gamma_{1,n},\gamma_{2,n}))=0.
$$
So we have $s_n\rightarrow 0$ and
$\tilde u_{i,n}^{-}\rightarrow 0$ in $H_i$, $i=1,2$, where $u^-=\min\{-u,0\}$, then take
\begin{equation}\label{eq4-7-1}
(u_{1,n},u_{2,n})=(s_n\star \tilde u_{1,n},s_n\star \tilde u_{2,n}),
\end{equation}
 we have followings.
\begin{lemma}\label{le4-7-2}
$(u_{1,n},u_{2,n})$ which obtained in \eqref{eq4-7-1} is a PS sequence of $J|_{S_1\times S_2}$ at level $c$.
\end{lemma}
\begin{proof}
First from above we have $(s_n,\tilde u_{1,n},\tilde u_{2,n})$ is a PS sequence at level $c$. Then for any $(\phi_1,\phi_2)\in H_1\times H_2$
\begin{equation}
\begin{split}
&\tilde J'_{\textbf{u}}(s_n,\tilde u_{1,n},\tilde u_{2,n})(\phi_1,\phi_2)\\
={}&e^{2s_n}\int_{\mathbb{R}^3}\nabla\tilde u_{1,n}\cdot\nabla\phi_1+\nabla \tilde u_{2,n}\cdot\nabla \phi_2+\int_{\mathbb{R}^3}V_1(e^{-s_n}x)\tilde u_{1,n}\phi_1+V_2(e^{-s_n}x)\tilde u_{2,n}\phi_2\\
&-e^{3s_n}\int_{\mathbb{R}^3} \mu_1\tilde u_{1,n}^3\phi_1+\mu_2\tilde u_{2,n}^3+\beta \tilde u_{1,n}\phi_1\tilde u_{2,n}^2+\beta\tilde u_{1,n}^2\tilde u_{2,n}\phi_2\\
={}&\int_{\mathbb{R}^3} \nabla u_{1,n}\cdot\nabla (s_n\star \phi_1)+\nabla u_{2,n}\cdot\nabla (s_n\star \phi_2)\\
&+\int_{\mathbb{R}^3}V_1(x)u_{1,n}(s_n\star \phi_1)+V_2(x)u_{2,n}(s_n\star \phi_2)\\
&-\int_{\mathbb{R}^3} \mu_1 u_{1,n}^3(s_n\star \phi_1)+\mu_2 u_{2,n}^3(s_n\star \phi_1)+\beta u_{1,n}^2 u_{2,n}(s_n\star \phi_2)+\beta u_{1,n}u_{2,n}^2(s_n\star \phi_1)\\
={}&J'(u_{1,n},u_{2,n})(s_n\star \phi_1,s_n\star \phi_2),\nonumber
\end{split}
\end{equation}
notice that $-s_n\star (s_n\star \phi)=\phi$, we have

$$\tilde J'_{\textbf{u}}(s_n,\tilde u_{1,n},\tilde u_{2,n})(-s_n\star\phi_1,-s_n\star\phi_2)= J'(u_{1,n},u_{2,n})( \phi_1, \phi_2).$$
It is obvious that $(\phi_1,\phi_2)\in T_{(u_{1,n},u_{2,n})}S_1\times S_2$ if and only if $(-s_n\star \phi_1,-s_n\star\phi_2)\in T_{(\tilde u_{1,n},\tilde u_{2,n})}S_1\times S_2$ the proof can be find in \cite{J-1}. Because $s_n\rightarrow 0$ we have  $-s_n\star \phi_i\rightarrow\phi_i$, $i=1,2$ as $n\rightarrow \infty$ in $H^1$, then for $n$ large enough we have there exist $A_1>0$ and $A_2>0$ such that
$$A_1<\frac{\|(\phi_1,\phi_2)\|}{\|(-s_n\star \phi_1,-s_n\star \phi_2)\|}<A_2,$$
where $(\phi_1,\phi_2)\neq (0,0)$. Let $\|\cdot\|_\star$ be the norm of the cotangent space $(T_{(u_{1},u_{2})}S_1\times S_2)^\star$ then we have
\begin{equation}
\begin{split}
| J|'_{S_1\times S_2}(u_{1,n},u_{2,n})\frac{(\phi_1,\phi_2)}{\|(-s_n\star \phi_1,-s_n\star \phi_2)\|}|&\leq \|(\tilde J|_{S_1\times S_2} )'_{\textbf u}(s_n,\tilde u_{1,n},\tilde u_{2,n})\|_{\star}\rightarrow 0,\nonumber
\end{split}
\end{equation}
for any  $(\phi_1,\phi_2)\in T_{(u_{1,n},u_{2,n})}S_1\times S_2$ and $(\phi_1,\phi_2)\neq (0,0)$, as $n\rightarrow \infty$. Take the supremum both side we have
$$A_1\| J|'_{S_1\times S_2}(u_{1,n},u_{2,n})\|_{\star} \leq \|(\tilde J|_{S_1\times S_2} )'_{\textbf u}(s_n,\tilde u_{1,n},\tilde u_{2,n})\|_{\star}\rightarrow 0,~\hbox{as}~n\rightarrow \infty.$$
 Notice that $A_1>0$ we have $$\| J|'_{S_1\times S_2}(u_{1,n},u_{2,n})\|_{\star}\rightarrow 0,~\hbox{as}~n\rightarrow \infty.$$ On the other hand we have
$$J(u_{1,n},u_{2,n})=\tilde J(s_n,\tilde u_{1,n},\tilde u_{2,n})\rightarrow c,~\hbox{as}~n\rightarrow \infty,$$
then we finishes the proof.
\end{proof}
\begin{lemma}\label{le5-7-1}
Assume $(A_4)$ holds then the PS sequence $(u_{1,n},u_{2,n})$ is bounded.
\end{lemma}
\begin{proof}
because $(s_n,\tilde u_{1,n},\tilde u_{2,n})$ is PS sequence of $\tilde J_{\mathbb{R}\times S_1\times S_2}$ then we have $$\frac{d}{ds}\tilde J(s_n,\tilde u_{1,n},\tilde u_{2,n})\rightarrow 0$$
i.e.
\begin{equation}\label{eq4-7-2}
\begin{split}
&\int_{\mathbb{R}^3} |\nabla u_{1,n}|^2+|\nabla u_{2,n}|^2-\frac{3}{4}\int_{\mathbb{R}^3} \mu_1 u_{1,n}^4+\mu_2 u_{2,n}^4+2\beta u_{1,n}^2u_{2,n}^2\\
&-\int_{\mathbb{R}^3} \nabla V_1(e^{-s_n}x)\cdot e^{-s_n}x\tilde u_{1,n}^2-\int_{\mathbb{R}^3} \nabla V_2(e^{-s_n}x)\cdot e^{-s_n}x\tilde u_{2,n}^2\rightarrow 0,
\end{split}
\end{equation}
on the other hand notice that $\tilde J(s_n,\tilde u_{1,n},\tilde u_{2,n})=J(u_{1,n},u_{2,n})$, we have
\begin{equation}\label{eq4-7-3}
\begin{split}
J(u_{1,n},u_{2,n})=&\frac{1}{2}\int_{\mathbb{R}^3} |\nabla u_{1,n}|^2+|\nabla u_{2,n}|-\frac{1}{4}\int_{\mathbb{R}^3} \mu_1 u_{1,n}^4+\mu_2 u_{2,n}^4+2\beta u_{1,n}^2u_{2,n}^2\\
&+\int_{\mathbb{R}^3}V_1(x)u_{1,n}^2+V_2(x)u_{2,n}^2\rightarrow c.
\end{split}
\end{equation}
From \eqref{eq4-7-2} and \eqref{eq4-7-3}, we have
\begin{equation}
\begin{split}
&\frac{1}{2}\int_{\mathbb{R}^3} |\nabla u_{1,n}|^2+V_1(x)u_{1,n}^2+\frac{1}{2}\int_{\mathbb{R}^3}|\nabla u_{2,n}|^2+V_2(x)u_{2,n}^2\\
{}&+\int_{\mathbb{R}^3} (2V_1(x)+\nabla V_1(x)\cdot x) u_{1,n}^2+\int_{\mathbb{R}^3} (2V_2(x)+\nabla V_2(x)\cdot x) u_{2,n}^2\rightarrow 3c.\nonumber
\end{split}
\end{equation}
Then by condition $(A_4)$, we have $(u_{1,n},u_{2,n})$ is bounded in $H_1\times H_2$.
\end{proof}
\begin{proof}[Proof of Theorem \ref{th3-30-1}]
From Lemma \ref{le5-7-1}, we have that $(u_{1,n},u_{2,n})\in S_1\times S_2$ is bounded in $H_1\times H_2$, then there exists $(u_{1,0},u_{2,0})$ is nonegative such that
$$(u_{1,n},u_{2,n})\rightharpoonup (u_{1,0},u_{2,0})~\hbox{in}~H_1\times H_2,$$
$$(u_{1,n},u_{2,n})\rightarrow (u_{1,0},u_{2,0})~\hbox{in}~L^2\times L^2.$$
From Lemma \ref{le4-7-2}, we have $J|_{S_1\times S_2}'(u_{1,n},u_{2,n})\rightarrow 0$ as $n\rightarrow \infty$, then for any $(\phi_1,\phi_2)\in H_1\times H_2$, there exist sequences $(\lambda_{1,n})$ and $(\lambda_{2,n})$ such that

\begin{equation}\label{3-14}
\begin{split}
&(J|'_{S_1\times S_2}(u_{1,n},u_{2,n}),(\phi_1,\phi_2))\\
={}&\int_{\mathbb{R}^3} \nabla u_{1,n}\nabla \phi_1+\int_{\mathbb{R}^3} \nabla u_{2,n}\nabla \phi_2-\mu_1\int u_{1,n}^3\phi_1-\mu_2\int u_{2,n}^3\phi_2\\
&-\beta\int_{\mathbb{R}^3} u_{1,n}u_{2,n}^2\phi_1-\beta\int_{\mathbb{R}^3} u_{1,n}^2u_{2,n}\phi_2+\int_{\mathbb{R}^3}V_1(x)u_{1,n}\phi_1\\
&+\int_{\mathbb{R}^3}V_2(x)u_{2,n}\phi_2-\lambda_{1,n}\int_{\mathbb{R}^3} u_{1,n}\phi_1-\lambda_{2,n}\int_{\mathbb{R}^3} u_{2,n}\phi_2\\
={}&o(\|(\phi_1,\phi_2)\|).
\end{split}
\end{equation}
Taking $(\phi_1,\phi_2)=(u_{1,n},0)$ and $(\phi_1,\phi_2)=(0,u_{2,n})$ in \eqref{3-14}, notice that $(u_{1,n},u_{2,n})$ is bounded, we have
\begin{equation}
\lambda_{1,n}a_1^2=\int_{\mathbb{R}^3} |\nabla u_{1,n}|^2+\int_{\mathbb{R}^3} V_1(x) u_{1,n}^2-\mu_1\int_{\mathbb{R}^3} u_{1,n}^4-\beta \int_{\mathbb{R}^3} u_{1,n}^2u_{2,n}^2+o(1),\nonumber
\end{equation}
\begin{equation}
\lambda_{2,n}a_2^2=\int_{\mathbb{R}^3} |\nabla u_{2,n}|^2+\int_{\mathbb{R}^3} V_2(x) u_{2,n}^2-\mu_2\int_{\mathbb{R}^3} u_{2,n}^4-\beta \int_{\mathbb{R}^3} u_{1,n}^2u_{2,n}^2+o(1).\nonumber
\end{equation}
Moreover, we have $(\lambda_{1,n})$ and $(\lambda_{2,n})$ are bounded, then we may assume there exist $\lambda_1$ and $\lambda_2$ in $\mathbb{R}$, such that
$$\lambda_{1,n}\rightarrow \lambda_1,$$
$$\lambda_{2,n}\rightarrow \lambda_2,$$
then from standard weak convergence argument, we have $(\lambda_1,\lambda_2,u_{1,0},u_{2,0})$ is a nonegative solution of \eqref{yun3-10-1}-\eqref{yun3-10-2}.
\end{proof}
\section{The case $\kappa\neq0$}
We will prove Theorem \ref{th3-30-2}, Theorem \ref{th3-30-3}, Theorem \ref{th3-30-4} and Theorem \ref{th3-30-5} by minimax method which is similar to Section 2.\par
First we prove  Theorem \ref{th3-30-2} and Theorem \ref{th3-30-3}.
\begin{lemma}\label{le4-9-1}
Assume \eqref{6-1-1} and $|\kappa|<\frac{4}{27C_{a_1,a_2}^2a_1a_2}$ hold, then we have
$$\inf_{B_K}J>0,$$
and
$$\inf_{B_K}J>\sup_{A_K}J,$$
by taking $\theta=3$ and $K=\frac{8^2}{9^3C^2_{a_1,a_2}}.$
\end{lemma}
\begin{proof}
 Taking $(v_1,v_2)\in B_K$ and $(u_1,u_2)\in A_K$ there exists $\delta>0$ such that
\begin{equation}
\begin{split}
J(v_1,v_2)-J(u_1,u_2)\geq&\frac{1}{2}\int_{\mathbb{R}^3} |\nabla v_1|^2+V_1(x)v_{1}^2+|\nabla v_2|^2+V_2(x)v_{2}^2\\
&-\frac{1}{2}\int_{\mathbb{R}^3} |\nabla u_1|^2+V_1(x)u_{1}^2+|\nabla u_2|^2+V_2(x)u_{2}^2\\
&-\frac{1}{4}\int_{\mathbb{R}^3} \mu_1v_1^4+\mu_2 v_2^4+2\beta v_1^2 v_2^2+\frac{1}{4}\int_{\mathbb{R}^3} \mu_1 u_1^4+\mu_2 u_2^4+2\beta u_1^2u_2^2\\
&-2|\kappa| a_1a_2\\
\geq& \frac{1}{2}\int_{\mathbb{R}^3} |\nabla v_1|^2+V_1(x)v_{1}^2+|\nabla v_2|^2+V_2(x)v_{2}^2\\
&-\frac{1}{2}\int_{\mathbb{R}^3} |\nabla u_1|^2+V_1(x)u_{1}^2+|\nabla u_2|^2+V_2(x)u_{2}^2\\
&-\frac{1}{4}\int_{\mathbb{R}^3} \mu_1v_1^4+\mu_2 v_2^4+2\beta v_1^2 v_2^2-2|\kappa| a_1a_2\\
\geq& K-\frac{C_{a_1,a_2}}{4}(3K)^{\frac{3}{2}}-2|\kappa| a_1a_2\\
>&\delta,\nonumber
\end{split}
\end{equation}
the last inequality is because when $K=\frac{8^2}{9^3C_{a_1,a_2}^2}$,  the function
$$ K-\frac{C_{a_1,a_2}}{4}(3K)^{\frac{3}{2}}-2\kappa a_1a_2,$$
achieves its maximum value. Moreover, it is easy to verify $\inf_{B_K}J>0$, thus we finishes the proof.
\end{proof}
\begin{remark}
It is easy to see that condition \eqref{6-1-1} is in order to ensure that $A_K$ is not empty, when $K=\frac{8^2}{9^3C_{a_1,a_2}^2}$.
\end{remark}
\begin{proof}[Proof of Theorem \ref{th3-30-2}] First we fix a point $(v_1,v_2)$ which is nonnegative and take $(w_1,w_2):=(s*v_1,s*v_2)$ then we have $J(w_1,w_2)$ is negative enough when $s\in\mathbb{R}$ large enough. Next following the proof of Theorem \ref{th3-30-1} and note that $J(|u_1|,|u_2|)\leq J(u_1,u_2)$, then there exists $(\lambda_1,\lambda_2,u_{1,0},u_{2,0})$ is a solution for \eqref{yun3-10-1}-\eqref{yun3-10-2}, and $u_{1,0}$ and ${u_{2,0}}$ are both nonnegative.
\end{proof}

\begin{proof}[Proof of Theorem \ref{th3-30-3}] Same as the proof of Theorem \ref{th3-30-2}, note that $J(-|u_1|,|u_2|)\leq J(u_1,u_2)$, $J(|u_1|,-|u_2|)\leq J(u_1,u_2)$.
\end{proof}

Next we turn to prove Theorem \ref{th3-30-4} and Theorem \ref{th3-30-5}. To this aim, we need to construct an other minimax structure of $J|_{S_1\times S_2}$. First we define following sets
$$\bar A_K:=\{(u_1,u_2)\in S_1\times S_2: \int_{\mathbb{R}^3} [|\nabla u_1|^2+V_1(x)u_1^2+|\nabla u_2|^2+V_2(x)u_2^2]-2\kappa\int_{\mathbb{R}^3}u_1u_2\leq K\},$$
$$\bar B_K:=\{(u_1,u_2)\in S_1\times S_2: \int_{\mathbb{R}^3} [|\nabla u_1|^2+V_1(x)u_1^2+|\nabla u_2|^2+V_2(x)u_2^2]-2\kappa\int_{\mathbb{R}^3}u_1u_2= \theta K\}.$$
\begin{remark}
Under the assumption \eqref{6-1-2} and $\kappa>0$ we have $\bar A_K$ is nonempty. Infact we take $u_1=a_1u_{V_1}$ and $u_2=a_2u_{V_2}$ where $u_{V_i}$, $i=1,2$ are the minimal point of $\lambda_{V_i}$, $i=1,2$. On the other hand when $\kappa<0$ we have $\bar A_K$ is nonempty. infact we take $u_1=a_1u_{V_1}$ and $u_2=-a_2u_{V_2}$ where $u_{V_i}$, $i=1,2$ is the minimal point of $\lambda_{V_i}$, $i=1,2$.
\end{remark}
Because $0<\kappa<\sqrt{\inf_{\mathbb{R}^3}V_1\cdot\inf_{\mathbb{R}^3}V_2}$ holds then
$$\int_{\mathbb{R}^3}V_1(x)u_1^2+V_2(x)u_2^2-\kappa u_1u_2> 0.$$
Moreover, same as \eqref{eq3-30-1} we have following estimate
\begin{equation}\label{7-24-1}
\begin{split}
&\int_{\mathbb{R}^3}\mu_1 u_1^4+\mu_2 u_2^4+2\beta u_1^2u_2^2\\
\leq & C_{a_1,a_2}(\int_{\mathbb{R}^3}|\nabla u_1|^2+|\nabla u_2|^2)^{\frac{3}{2}}\\
 \leq & C_{a_1,a_2}(\int_{\mathbb{R}^3}|\nabla u_1|^2+|\nabla u_2|^2+V_1(x)u_1^2+V_2(x)u_2^2-\kappa u_1u_2)^{\frac{3}{2}}.
\end{split}
\end{equation}
Similarly to Lemma \ref{le4-7-1} we have following lemma.
\begin{lemma}
Assume \eqref{6-1-2} holds, then there exists a $K>0$ such that
$$\inf_{\bar B_K}J>0,$$
and
$$\inf_{\bar B_K}J>\sup_{\bar A_K}	J.$$
\end{lemma}
\begin{proof}
We take $\lambda_{V_1}a_1^2+\lambda_{V_2}a_2^2<K<\frac{16}{27C_{a_1,a_2}^2}$ and $\theta=3$, same as Lemma \ref{le4-7-1}, first we take $(u_1,u_2)\in \bar B_K$, by G-N inequality, there exists $\delta_1>0$ such that
\begin{equation}
\begin{split}
J(u_1,u_2)=&\frac{1}{2}\int_{\mathbb{R}^3} |\nabla u_1|^2+V_1(x)u_1^2+\frac{1}{2}\int_{\mathbb{R}^3}|\nabla u_2|^2+V_2(x)u_2^2-\kappa\int_{\mathbb{R}^3}u_1u_2\\
&-\frac{1}{4}\int_{\mathbb{R}^3}\mu_1 u_1^4+\mu_2 u_2^4+2\beta u_1^2u_2^2\\
\geq &\frac{3}{2}K-\frac{C_{a_1,a_2}}{4}(3K)^{\frac{3}{2}}\\
>&\delta_1.\nonumber
\end{split}
\end{equation}
Next, we take $(v_1,v_2)\in \bar B_K$ and $(u_1,u_2)\in \bar A_K$, together with \eqref{7-24-1}, there exists $\delta_2>0$ such that
\begin{equation}
\begin{split}
J(v_1,v_2)-J(u_1,u_2)=&\frac{1}{2}\int_{\mathbb{R}^3} |\nabla v_1|^2+V_1(x)v_{1}^2+|\nabla v_2|^2+V_2(x)v_{2}^2-2\kappa v_1v_2\\
&-\frac{1}{2}\int_{\mathbb{R}^3} |\nabla u_1|^2+V_1(x)u_{1}^2+|\nabla u_2|^2+V_2(x)u_{2}^2-2\kappa u_1u_2\\
&-\frac{1}{4}\int_{\mathbb{R}^3} \mu_1v_1^4+\mu_2 v_2^4+2\beta v_1^2 v_2^2+\frac{1}{4}\int_{\mathbb{R}^3} \mu_1 u_1^4+\mu_2 u_2^4+2\beta u_1^2u_2^2\\
\geq& \frac{1}{2}\int_{\mathbb{R}^3} |\nabla v_1|^2+V_1(x)v_{1}^2+|\nabla v_2|^2+V_2(x)v_{2}^2-2\kappa v_1v_2\\
&-\frac{1}{2}\int_{\mathbb{R}^3} |\nabla u_1|^2+V_1(x)u_{1}^2+|\nabla u_2|^2+V_2(x)u_{2}^2-2\kappa u_1u_2\\
&-\frac{1}{4}\int_{\mathbb{R}^3} \mu_1v_1^4+\mu_2 v_2^4+2\beta v_1^2 v_2^2\\
\geq& K-\frac{C_{a_1,a_2}}{4}(3K)^{\frac{3}{2}}\\
>&\delta_2,\nonumber
\end{split}
\end{equation}
which finishes the proof.
\end{proof}
Next we fix $(v_1,v_2)\in \bar A_K$ and take $s\in\mathbb{R}$ large enough, then the norm of $(w_1,w_2)=(s\star v_1,s\star v_2)$ is big enough and $J(w_1,w_2)$ is negative enough, thus any path from $(v_1,v_2)$ to $(w_1,w_2)$ must pass through $\bar B_K$, which is the same reason in Section 2.

\begin{proof}[Proof of Theorem \ref{th3-30-4}] First we construct a minimax structure of $J|_{S_1\times S_2}$
\begin{equation}
\Gamma:=\{\gamma(t)=(\gamma_1(t),\gamma_2(t))\in C([0,1],S_1\times S_2):\gamma(0)=(v_1,v_2),\gamma(1)=(w_1,w_2),t\in [0,1]\},\nonumber
\end{equation}
together with its minimax value
$$c:=\inf _{\gamma \in \Gamma}\sup_{t\in [0,1]}J(\gamma(t))\geq \inf_{\bar{B}_K}J(u_1,u_2)>0.$$
Then we follow the procedure in Section 2, we can get Theorem \ref{th3-30-4}.
\end{proof}

\begin{proof}[Proof of Theorem \ref{th3-30-5}] It is same as the proof of Theorem \ref{th3-30-3} and we omit it.
\end{proof}

\end{document}